\documentclass{amsart}
\usepackage[T1]{fontenc}
\usepackage[latin9]{inputenc}
\usepackage[english]{babel}
\usepackage{amssymb}
\usepackage{shltx}
\usepackage{shthm}
\usepackage[sc]{mathpazo}
\def\restrict#1#2
  {\left.#2\right\vert_{#1}}

\def\labelledto#1
  {\mathop{\relbar\joinrel\relbar\joinrel\relbar\joinrel\rightarrow}\limits^{#1}}

\newcommand\id{\mathop{\mathrm{id}}}

\newcommand\GSL{\mathbf{SL}}		
\newcommand\Grass{\mathbf{Gr}}		
\newcommand\GP{\mathbf{P}}

\newcommand\Orbit{\mathcal{O}}		
\newcommand\Nilpotent{\mathcal{N}}	

\newcommand\Regular{\mathcal{O}}	
\newcommand\Bouquet{\mathcal{X}/(F,h)}

\newcommand\Ga{\mathfrak{a}}
\newcommand\Gb{\mathfrak{b}}
\newcommand\Gm{\mathfrak{m}}		
\newcommand\Gg{\mathfrak{g}}
\newcommand\Gh{\mathfrak{h}}
\newcommand\Gs{\mathfrak{s}}
\newcommand\Gu{\mathfrak{u}}

\newcommand\Lie{\mathop{\mathcal{L}}}

\newcommand\SGrp{\Gamma}		
\newcommand\SGrs{\Grass(k,\Gg)}		
\newcommand\SLie{\Lambda}		
\newcommand\SAlg{\Alpha}		

\newcommand\UFam{\mathfrak{H}}		
\newcommand\Ufam{\mathfrak{h}}		
\newcommand\UMod{\mathfrak{P}}		

\newcommand\FFam{\mathcal{H}}
\newcommand\FMod{P}

\newcommand\VFam{\mathcal{V}}		

\newcommand\Tautological{\mathcal{T}}

\begin{document}
\author{Michaël Le Barbier Grünewald}
\address{%
Michaël Le Barbier Grünewald\\
Max-Planck-Institut für Mathematik\\
Vivatsgasse 7\\
$53\,111$ Bonn
}
\email{michi@mpim-bonn.mpg.de}
\urladdr{http://guests.mpim-bonn.mpg.de/michi}
\keywords{Moduli Space, Complex Algebraic Groups}
\title{Universal family of the subgroups of an algebraic group}
\date{30 October 2008}
\begin{abstract}
We construct a moduli space for the connected subgroups of an
algebraic group and the corresponding universal family. Morphisms from
an algebraic variety to this moduli space correspond to flat families
of connected algebraic subgroups parametrised by this variety. This
moduli space is obtained by gluing together infinitely many
irreducible projective varieties of bounded dimension along closed subvarieties.

Regarding families of non-connected subgroups of an algebraic group,
we show that, given sich a family, the corresponding family of
identity components is an irreducible component of the former, and the
quotient of a family of groups group by the family of their identity
components exists.
\end{abstract}

\maketitle
\section{Introduction}
Let~$G$ be a complex connected algebraic group,
a flat family of subgroups of~$G$ parametrised by a variety~$\FMod$ is a
subvariety~$\FFam\subset G \times \FMod$ such that the
projection~$\FFam\to\FMod$ is a flat morphism and whose scheme-theoretic
fiber~$\FFam_p$ at any~$p\in\FMod$ is a group subscheme of~$G$.
We construct a moduli space and a universal family for the connected subgroups
of a complex algebraic group and discuss a few examples: tori, abelian
varieties, $\GSL_2$ and~$\GSL_3$.

\subsubsection{Notations}
Let~$G$ be a complex connected algebraic group, and~$\Gg$ its Lie algebra. We
fix an integer~$k$ and denote by~$\SGrp$ the set of~$k$-dimensional connected
subgroups of~$G$. A connected subgroup~$H$ of~$G$ is determined by its Lie
algebra~$\Lie(H)$, as a subalgebra of~$\Gg$, so we denote by~$\SGrs$ the
Grassmann variety of~$k$-planes in~$\Gg$, by~$\SLie \subset \SGrs$ the variety
of $k$-dimensional Lie subalgebras of~$\Gg$ and by~$\SAlg$ the set
\[
\SAlg = \set{ \Lie(H) \mid H \in \SGrp }
\]
of \emph{algebraic} $k$-dimensional Lie subalgebras of~$G$. The terminology
\emph{algebraic} was introduced by~Chevalley~\cite{CHEVALLEYLIE2}. We will take
advantage in our construction of the natural operation of~$G$ on these sets.

\subsubsection{Main results and structure}
It belongs to moduli space problems that
we have to allow more general objects
than algebraic varieties in order to obtain a satisfying construction. Namely,
we introduce the category of \emph{bouquets of algebraic
varieties}~(\ref{ss:mobouquet}), that contains all algebraic varieties and
objects obtained by gluing together possibly infinitely many algebraic
varieties along closed subvarieties. These bouquets are similar to
\emph{ind-varieties} introduced by~Shafarevich~\cite{IND1} and~Kumar~\cite{IND2}.

We describe in Section~\ref{se:mo} a structure~$\UMod$ of bouquet on~$\SGrp$ and
show that the set-theoretic universal family
\[
\UFam = \set{ (H,g) \in G\times\UMod \mid g \in H }
\]
is actually a bouquet-theoretic universal
family~(\ref{th:main2}). The irreducible components of~$\UMod$ are projective
varieties and the ones of~$\UFam$ are quasi-projective.

\begin{theorem*}[\ref{th:main2}]
Let~$\FMod$ be an algebraic variety, and denote by
\begin{enumerate}
\item
$F$ the set of flat families of $k$-dimensional connected subgroups of~$G$ parametrised
by~$\FMod$;
\item
$M$ the set of morphisms from~$\FMod$ to~$\UMod$.
\end{enumerate}
Then the map~$F\to M$ sending a family~$\FFam$ to the morphism to~$\SLie$
mapping~$p\in\FFam$ to~$\Lie(\FFam_p)$ and the map from~$M \to F$
sending a morphism~$\psi$ to the family~$\psi^*\UFam$ define a natural
correspondance between~$F$ and~$M$.
\end{theorem*}

The bouquet structure on~$\UMod$ is \emph{a posteriori} defined
by the fact that~$\UFam$ is a final object in the category of flat
families of~$k$-dimensional subgroups of~$G$. The Lie functor~$\Lie$
puts~$\SGrp$ and~$\SAlg$ in a one-to-one correspondence, but
the topology induced by~$\SGrs$ on~$\SGrp$ is
much coarser than it needs to be, with respect to the universal
property of~$\UFam$. From the three stages we need to go through
before get to the proof of Theorem~\ref{th:main2}---the construction of the moduli
space~$\UMod$, of the universal family~$\UFam$ and the proof of the
universal property---the second one draws most of our efforts. For this we
need to show the

\begin{theorem*}[\ref{th:main1}]
Let~$\FMod$ be an irreducible constructible subset of~$\SLie$. If~$\FMod$ is
contained in~$\SAlg$, then the closure~$\bar\FMod$ of~$\FMod$ in~$\SLie$ is
also contained in~$\SAlg$ and~$\UFam(\bar\FMod)$ equals~$\bar\UFam(\FMod)$.
\end{theorem*}

Note that restriction of the
tautological family~$\UFam(\bar\FMod)$ mentioned in the statement of our
Theorem~\ref{th:main1} is flat. The tautological family can be described
through to the exponential map of~$G$:

\begin{corollary*}[\ref{co:main1}]
The set~$\exp\times\id_{\SLie}(\Ufam(\bar\FMod))$ is
actually~$\UFam(\bar\FMod)$, it is in particular algebraic.
\end{corollary*}

This Corollary is actually equivalent to Theorem~\ref{th:main1}, and we first
thought at it as an intermediary result on the road to
Theorem~\ref{th:main1}. However we were not able to overcome the difficulties
tied to the analytic nature of the exponential map, and had to put things
the other way around. Section~\ref{se:di} is devoted to its proof,
that is built on two principles: First we do not need to show that all
of~$\UFam(\bar\FMod)$ is an algebraic variety but we are allowed to
work on a smaller set of parameters (see~\ref{ss:intgr}). Second we
observe that some discrete invariants of groups behave
semi-continuously in families (see~\ref{ss:malcev}
and~\ref{ss:resolv}). This observation allows us to find the suitable
small subset of parameters over which we are able to prove the
algebraicity of the set-theoretic family.

\medskip
While representation theory of algebraic groups provide us with a
wealth of examples of families of algebraic groups (see
Example~\ref{ex:gform}) it is unwise to expect them having irreducible
fibres (see Example~\ref{ex:reduced}) even in the case of the family
is a connected variety. In Section~\ref{se:ic}, we first show that the
set-theoretic family of subgroups of~$G$ deduced from an algebraic one
by replacing each fibre by its identity component is algebraic as
well~\ref{co:clcc}. Second, we show how to construct a quotient of a
family of subgroups by the family of its identity
components~\ref{co:stein}.

\subsubsection*{Acknowledgement}
The work presented here was done during my stay at the Max-Planck-Institut für
Mathematik. It is to me a pleasure to thank this kind institution.


\tableofcontents
\section{Elementary properties of families of subgroups}
\subsubsection{Definitions and examples}

The definition of family of subgroups that we introduce is a close variation
on the definition of families of subvarieties
(see~\cite[III.9.10]{H}). However, we allow disconnected fibres. We then show
two results allowing us to construct new families of subgroups, by pushing a
family forward or backward through a morphism or intersecting its fibres with a fixed
group. Less elementary construction theorems must be deferred until we have
gathered the necessary prerequisites~(section~\ref{se:di}).

\begin{definition}
Let~$X$ and~$\FMod$ be two algebraic varieties.
A family of subschemes of~$X$ parametrised by~$\FMod$ is a
subvariety~$\VFam$ of~$X \times \FMod$ such that the morphism~$\pi: \VFam \to
\FMod$ obtained by restricting the projection of~$X\times \FMod$
on~$\FMod$ to~$\VFam$ is surjective. The fibre~$\VFam_p$ of the family at~$p$
is the subscheme~$\inv\pi(p)$ of~$X$.

The family is \emph{irreducible} when~$\VFam$ is irreducible.
The family is~\emph{reduced} at~$p\in\FMod$ if the maximal ideal~$\Gm_p$
of~$p\in\FMod$ generates the ideal of~$\VFam_p$. The family is reduced when all
its fibre are reduced. The family is~\emph{flat} when the morphism~$\pi$ is
flat.
\end{definition}

We often refer to a subset~$\VFam$ of~$X \times\FMod$ under the term
``set-theoretic family'' when it needs not to be a
subvariety of~$G\times\FMod$. In this context we refer to the families defined above under
the term ``algebraic family.''

\begin{definition}
Let~$G$ be an algebraic group and~$\FMod$ a variety. A family~$\FFam$ of
loose subgroups of~$G$ is a family of subschemes of~$G$ such that each
fibre~$\FFam_p$ is supported by an algebraic subgroup of~$G$. A family of
subgroups of~$G$ is a loose family whose fibres are reduced.
\end{definition}

A family~$\FFam$ of loose subgroups of~$G$ parametrised by~$\FMod$ is trivial
when it is of the form~$H \times P$, where~$H$ is a subgroup of~$G$. Such a
family is reduced.

We denote by~$\iota: G \to G$ the inverse map of~$G$, and by~$\mu:G\times G \to G$
its product map. These morphisms extend to families: we define
\[
\fun \iota_\FMod\:
G \times \FMod \to G\times\FMod\\
(g,p)\mapsto(\iota(g), p)\\
\qquad
\text{and}
\qquad
\fun \mu_\FMod\:
G \times G \times \FMod\to G\times\FMod\\
(g_1,g_2,p)\mapsto(\mu(g_1,g_2), p)
.
\\
\]
If~$\FFam$ is a family of
loose subgroups of~$G$ parametrised by a variety~$\FMod$, we denote by~$\FFam
\times_\FMod \FFam$ the fibred Cartesian square of~$\FFam$, it is a family of
loose subgroups of~$G \times G$ parametrised by~$\FMod$ whose fibre at~$p$
is~$\FFam_p\times\FFam_p$. The morphism~$\iota_\FMod$ restricts
to~$\FFam\to\FFam$ and the morphism~$\mu_\FMod$
to~$\FFam\times_\FMod\FFam\to\FFam$.

\begin{example}[General form of families]
\label{ex:gform}
Let~$X$ and~$\FMod$ be algebraic varieties and~$\VFam$ a family of subschemes
of~$X$ parametrised by~$\FMod$. The stabiliser~$G_\VFam$ of~$\VFam$
in~$G$ defined by
\[
G_\VFam = \set {
 (g,p)\in G \times \FMod
 \mid
   \forall x\in \VFam_p \quad
   gx\in \VFam_p
}
\]
is a family of subgroups of~$G$ parametrised by~$\FMod$ and each family
of subgroups of~$G$ can be presented as a stabiliser.
\end{example}

\begin{example}[A connected and reducible family]
\label{ex:reduced}
Let~$T$ be a torus in~$G = \SL_2$ and~$U$ a unipotent subgroup of~$\SL_2$ such
that~$TU$ is a Borel subgroup of~$\SL_2$. Let also
\[
\P^1 = U \cup \set{0,\infty}
\]
be a projective line compactifying~$U$. Consider the operation of~$\SL_2$ on
itself defined by conjugation and construct the closure~$\VFam$ in~$G \times
\P^1$ of
\[
\VFam^o = \set { (u t\inv u, u) \in G \times \P^1 \mid (t,u)\in T \times U}
. 
\]
For each~$u\in U$ the fiber~$\VFam_u$ is a torus of~$G$ and
at~$u\in\set{0,\infty}$ the fibre is a Borel subgroup of~$G$. Now let~$\FFam =
G_\VFam $ be the stabiliser of~$\VFam$ in~$G$. At each~$u\in U$ the
fibre~$\FFam_u$ is the normaliser of the torus~$\VFam_u$ in~$G$, that is an
extension of the Weyl group of~$G$ by~$\VFam_u$. It has two connected
components. At~$u\in\set{0,\infty}$ the fibre~$\FFam_u$ is the normaliser of
the Borel subgroup~$\VFam_u$ of~$G$, thus~$\FFam_u = \VFam_u$. Hence~$\VFam$
has two irreducible components that dominates~$\P^1$. We show that is possible
to do the quotient of a group by its identity component in
families~(\ref{co:stein}). Here the quotient space consists of two copies
of~$\P^1$ glued together along~$\set{0,\infty}$.
\end{example}

\begin{example}[A loose family]
We now give an example of a family of loose subgroups of~$\C^2$ that is not a
family of subgroups. Let~$C$ be the plane cuspidal cubic, whose algebra of
regular functions is
\[
\C[p_1,p_2]/(p_1^2 - p_2^3)
,
\]
and consider the family~$\FFam_0$ of lines in~$\C^2$
meeting~$C$ in~$(0,0)$ and a distinct point:
\[
\FFam_0 = \set{
  (x,y,p) \in \C^2\times C\setminus\set{(0,0)}
  \mid p_1 y - p_2 x = 0
}
.
\]
The closure~$\FFam$ of~$\FFam_0$ in~$\C^2 \times C$ is defined by the ideal
generated by
\[
p_1 y - p_2 x
\quad\text{and}\quad
y^2 - p_1 x^2
.
\]
In this family, each fibre is supported by a line in~$\C^2$, but the fibre
at~$(0,0)$ fails to be reduced. Hence this family of loose subgroups of~$\C^2$
is not a family of subgroups of~$\C^2$.
\end{example}


\section{Differentiation and integration of families}
\label{se:di}
\subsubsection{Differentiation of families of groups}
\label{ss:didiff}

Let~$G$ be an algebraic group,~$\FMod$ an algebraic variety and let~$\FFam$ be
a family of subgroups of~$G$ parametrised by~$\FMod$. We construct the
corresponding family of Lie algebras. For this, we consider 
the trivial family~$\Gg \times \FMod$ as a subvariety of the tangent bundle to~$G
\times \FMod$.

\begin{definition}
The family~$\Lie(\FFam)$ of Lie algebras associated to~$\FFam$ is
\[
\Lie(\FFam) = T\FFam \cap (\Gg\times \FMod)
.
\]
\end{definition}

Serre~\cite{GAGA} studied the analytic structure of algebraic
varieties. We recall some of his results in
Appendix~\ref{se:flat}. They allow us to relate the flatness
of~$\FFam$ and of~$\Lie(\FFam)$ through the exponential map of~$G$:

\begin{proposition}
\label{pr:diff}
Let~$\FFam$ be a family of connected loose algebraic subgroups of~$G$
parametrised by~$\FMod$. This family is flat if, and only if,~$\Lie(\FFam)$ is
flat.
\end{proposition}

\begin{proof}
The exponential~$\exp:\Gg\to G$ induces a biholomorphism from a
neighbourhood~$W_0$ of~$0$ in~$\Gg$ and a neighbourhood~$W_1$ of the identity~$e$
element in~$G$. It follows from~\ref{co:appflat} that the flatness of~$\FFam$ at
each point of~$\FFam\cap W_1\times\FMod$ is equivalent to the flatness
of~$\Lie(\FFam)$ at each point of~$\Lie(\FFam)\cap W_0\times\FMod$.
\end{proof}

\begin{corollary}
\label{co:diff}
Let~$\FFam$ be a family of connected $k$-dimensional algebraic subgroups of~$G$
parametrised by~$\FMod$. If~$\FFam$ is flat, then there is a
morphism~$\psi:\FMod\to\SGrs$ such that~$\Lie(\FFam)$ is obtained by
pulling back the tautological bundle~$\Tautological\to\SGrs$
through~$\psi$.
\end{corollary}


\subsubsection{Integration of families of Lie algebras}
\label{ss:intgr}

We study the possibility of going backwards in the process of differentiating
a flat family of connected subgroups of~$G$. With the notations of the
previous section,~$\psi(\FMod)$ is a constructible subset of~$\SLie$
contained into the set~$\SAlg$ of algebraic subalgebras of~$\Gg$. We show that the
integration of a flat family of algebraic Lie algebras is possible, let us
introduce some notations before we state our main theorem.

For any subset~$\FMod$ of~$\SAlg$ we write~$\Ufam(\FMod)$ for the set-theoretic
family of subalgebras of~$\Gg$ obtained by restricting to~$\FMod$ the tautological
bundle~$\Tautological\to\SGrs$. For all~$p\in \SAlg$ we denote by~$\UFam_p$ the
connected subgroup of~$G$ whose Lie algebra is~$\Ufam_p$ and by~$\UFam(\FMod)$ the
set-theoretic family of connected subgroups of~$G$ defined by
\[
\UFam(\FMod) = \set{
  (g,p)\in G\times \FMod
  \mid
  g\in\UFam_p
}
.
\]
We can now state our main

\begin{theorem}
\label{th:main1}
Let~$\FMod$ be an irreducible constructible subset of~$\SLie$. If~$\FMod$ is
contained in~$\SAlg$, then the closure~$\bar\FMod$ of~$\FMod$ in~$\SLie$ is
also contained in~$\SAlg$ and~$\UFam(\bar\FMod)$ equals~$\bar\UFam(\FMod)$.
\end{theorem}

This theorem follows from Proposition~\ref{pr:fconstr2} and
Proposition~\ref{pr:digeneral} below. For now, we
give two corollaries of~\ref{th:main1}:

\begin{corollary}
With the notations of~\ref{th:main1}, $\UFam(\bar\FMod)$ is a flat family of
connected subgroups of~$G$ whose corresponding family of Lie subalgebras
of~$\Gg$ is~$\Lie(\UFam(\bar\FMod)) = \Ufam(\bar\FMod)$.
\end{corollary}

\begin{proof}
We have~$\Lie(\UFam(\bar\FMod)) = \Ufam(\bar\FMod)$, but the right hand side
is the restriction to~$\bar\FMod$ of the tautological
bundle~$\Tautological\to\SGrs$, so it is a flat family of linear subspaces
of~$\Gg$. By~\ref{co:appflat}, the corresponding family of groups is flat.
\end{proof}

\begin{corollary}
\label{co:main1}
The set~$\exp\times\id_{\SLie}(\Ufam(\bar\FMod))$ is
actually~$\UFam(\bar\FMod)$, it is in particular algebraic.
\end{corollary}

\begin{proof}
For complex connected algebraic groups, the exponential map is always surjective, so
this result follows from~\ref{th:main1}.
\end{proof}


\subsubsection{Proof of~\ref{th:main1} ---  Construction of families}
We present some elementary constructions of families of groups.

\begin{proposition}
\label{pr:fgroup}
Let~$\FMod$ be an irreducible algebraic variety and~$\FFam \subset G
\times\FMod$ a family of subschemes of~$G$ parametrised by~$\FFam$.
If there exists an open
subset~$V$ of~$\FMod$ such that the restriction of~$\FFam$ to~$V$ is a
family of loose subgroups of~$G$, then~$\FFam$ itself is a family of
loose subgroups of~$G$.
\end{proposition}

\begin{proof}
\newcommand\FFamV{\FFam(V)}
\newcommand\FS[2]{{#2}\times_{#1}{#2}}
\newcommand\FSFFam{\FS \FMod\FFam}%
\newcommand\FSFFamV{\FS \FMod\FFamV}%
Let~$\theta: G \times G \times \FMod \to G \times \FMod$ be the morphism
defined by~$\theta(g_1, g_2, p) = (g_1 \inv g_2, p)$. Note that a
family~$\FFam \subset G \times \FMod$ of subvarieties of~$G$ is a family
of subgroups of~$G$ if, and only if,~$\theta(\FSFFam) \subset\FFam$.

Since~$\FFam$ is an irreducible subvariety of~$G\times\FMod$, it is
the closure in~$G\times\FMod$ of the
restriction~$\FFam(V)$ of~$\FFam$ to~$V$. The
relation~$\theta(\FSFFamV) \subset \FFamV$ and the continuity of~$\theta$
imply~$\theta(\FSFFam) \subset \FSFFam$. Hence~$\FFam$ is a family
of loose subgroups of~$G$.
\end{proof}

\begin{proposition}
\label{pr:fconstr}
Let~$\FMod$ be an irreducible algebraic variety and let~$\FFam\subset
G\times\FMod$ be a set-theoretic family of subgroups of~$G$. If~$\FFam$ is an
irreducible constructible subset of~$G\times\FMod$, then there is a dense
open subset~$V$ of~$\FMod$, such that the restriction of~$\FFam$ to~$V$ is an
algebraic family of loose subgroups of~$G$ parametrised by~$V$.
\end{proposition}

\begin{proof}
\newcommand\Vpg{V_{(p,g)}}%
\newcommand\Upg{U_{(p,g)}}%
Let~$U$ be the largest open subset of the closure of~$\FFam$ in~$\FMod
\times G$ contained in~$\FFam$. Since~$\FFam$ is constructible,~$U$ is
dense in the closure of~$\FFam$.

Take any fibre~$\FFam_p$ meeting~$U$. The set~$U_p = U \cap \FFam_p$ is
open and dense in~$\FFam_p$, thus for any~$g$ in~$\FFam_p$ the open dense
subsets~$\iota(U_p)g$ and~$U_p$ have a common point. This shows that~$g$
belongs to~$\mu(U_p\times U_p)$, that must be equal to~$ \FFam_p$. It implies
that~$\mu_\FMod(U \times_\FMod U)$ contains~$U$. But the morphism~$\mu_\FMod$
is flat, hence~$\mu_\FMod(U \times_\FMod U)$ is open
in~$\bar\FFam$~(\cite[III, ex.~9.1]{H}) and
contained in~$\FFam$. From the maximality of~$U$
follows that~$U = \mu_\FMod(U \times_\FMod U)$.

The set of parameters~$p$ corresponding to a fibre meeting~$U$ is
constructible and dense in~$\FMod$, hence it contains a dense open
subset~$V$ of~$\FMod$. The restriction of the set-theoretic family~$\FFam$
to~$V$ equals~$U \cap (G \times V)$, hence this family is algebraic.
\end{proof}

\begin{proposition}
\label{pr:fconstr2}
Let~$\FMod$ be a constructible subset of~$\SLie$ contained in~$\SAlg$.
If there is an open subset~$U$ of~$\bar\FMod$ contained in~$\FMod$ such
that~$\UFam(U)$ is algebraic, then~$\bar\FMod$ is contained in~$\SAlg$
and~$\UFam(\bar\FMod) = \bar \UFam(\FMod)$, in particular it is algebraic.
\end{proposition}

\begin{proof}
Let~$\VFam = \bar\UFam(U)$ the closure of~$\UFam(U)$
in~$G \times \SLie$. By~\ref{pr:fgroup} this is a family of loose
subgroups of~$G$. Considering the zero section of~$\UFam(U)$ makes it
clear that the projection of~$\VFam$ to~$\SLie$ is~$\bar\FMod$,
hence~$\VFam$ is parametrised by~$\bar\FMod$. We only have to show
that the Lie algebra of the reduced fibre of~$\VFam$ at~$p$
is~$\Ufam_p$, since this implies at once that~$p$ belongs to~$\SAlg$
and that~$\VFam_p = \UFam_p$.

To study~$\VFam$ closely we consider neighbourhoods~$W_0$ of~$0$
in~$\Gg$ and~$W_1$ of~$e$ in~$G$ that are biholomorphic through the
exponential map of~$G$. Now let~$p\in\bar\FFam\setminus U$
and~$x$ a closed point in~$\VFam_p \cap W_1$. Since~$x$ belongs to the
closure of~$\UFam(U)$, there is a curve~$C$ in~$\UFam(U)$
such that~$x\in\bar C$. Looking through the exponential we see
that~$\inv\exp(C\cap W_1)$ is contained in~$\Ufam(U)$. But the
family~$\Ufam$ is locally trivial and the closure of~$\Ufam(U)$
in~$\Gg\times\SGrs$ is~$\Ufam(\bar U) = \Ufam(\bar\FMod)$. It follows
that~$\inv\exp x$ belongs to~$\Ufam_p$ and that the Lie algebra of the
group supported by~$\VFam_p$ must be equal to~$\Ufam_p$.
\end{proof}

\begin{corollary}
\label{co:fconstr2}
Let~$\FMod$ be a constructible subset of~$\SLie$ contained in~$\SAlg$. The
set-theoretic family~$\UFam(\bar\FMod)$ of connected subgroups of~$G$ is algebraic
if the set-theoretic family~$\UFam(Q)$ is algebraic, where~$Q$ is
\begin{enumerate}
\item
either a dense open subset of~$\bar\FMod$;
\item
or the closure of~$G\cdot \FMod$ in~$\SLie$;
\item
or the intersection~$\bar\FMod \cap T$ of~$\bar\FMod$ with a subvariety~$T$
of~$\SLie$ such that
\[
\FMod \cap ( G\cdot(\FMod\cap T))
\]
is dense in~$\bar\FMod$.
\end{enumerate}
\end{corollary}

\begin{corollary}
\label{co:fconstr3}
Let~$G$ be an algebraic group, $\FMod$ a variety,
and~$\FFam^i\;(i\in\set{1,2})$ two families of loose subgroups of~$G$
parametrised by~$\FMod$. If for each~$p$ in~$\FMod$ the
group~$\FFam^1_p$ normalises~$\FFam^2_p$ and~$\FFam^1_p \cap\FFam^2_p$ is
finite, then the restriction of the set-theoretic family~$\FFam^1\FFam^2$
defined by
\[
(\FFam^1 \FFam^2)_p = \FFam^1_p \FFam^2_p
\]
to a suitable dense open subset of~$\FMod$ is algebraic.
\end{corollary}


\subsubsection{Proof of~\ref{th:main1} ---  Levi-Malcev decomposition}
\label{ss:malcev}

We study the behaviour of Levi-Malcev decomposition in families of Lie algebras.
The rigidity of semi-simple Lie algebras proved by
Richardson~\cite{RICHARDSON} implies the semi-continuity of the
application sending a Lie algebra to the isomorphism class of its semi-simple
part. Let us recall this rigidity theorem:

\newcommand\CiteRichardson{Richardson~\cite[9.2 and 9.6]{RICHARDSON}}
\begin{theorem}[\CiteRichardson]
Let~$\FMod$ be a subvariety of~$\SLie$ and~$p\in\FMod$. We denote by~$\Gs$ a
semi-simple Lie algebra contained in~$\UFam_p$. Then there is an analytic
neighbourhood~$W$ of~$p$ in~$\SLie$ and an analytic map~$\psi: W \to \Aut(\Gg)$
such that~$\psi(p)$ is the identical transformation and~$\psi(q)$ maps~$\Gs$
into~$\UFam_q$ for all~$q\in W$.
\end{theorem}

Let~$S$ be the set of isomorphism classes of semi-simple Lie algebras. We
denote the isomorphism class of a semi-simple Lie algebra~$\Gs$
by~$S(\Gs)$. The set~$S$ is countable, ordered by the relation
\[
S(\Gs_1) \le S(\Gs_2)
\quad
\iff
\quad
\text{there is an injection~$\Gs_1\to\Gs_2$}
.
\]
Each non-empty subset of~$S$ has at least one minimal element. For each Lie
algebra~$\Gh$ we denote by~$S(\Gh)$ the isomorphism class of a semi-simple
part of its Levi-Malcev decomposition.

\begin{proposition}
Let~$\FMod\subset\SLie$ be a locally closed subvariety of~$\SLie$. If~$S(\Gs)$
is a minimal element of
\[
\set{ S(\Ufam_p) \mid p \in \FMod }
,
\]
then the set
\[
\FMod' = \set{ p\in\FMod \mid S(\Ufam_p) = S(\Gs) }
\]
is Zariski closed in~$\FMod$.
\end{proposition}

\begin{proof}
Consider the variety
\[
\Hom_{\Lie}(\Gs, \Ufam(\FMod)) =
\set {
  (u,p) \in \Hom_{\Lie}(\Gs,\Gg) \times \FMod
  \mid
  \im u \subset \Ufam_p
}
.
\]
Since~$S(\Gs)$ is minimal along the isomorphism classes of the semi-simple
factors occuring in Levi-Malcev decomposition of fibres of~$\Ufam(\FMod)$, a
morphism from~$\Gs$ to~$\Ufam_p$ is either zero or injective. The set~$P'$ is
thus the projection on~$P$ of
\[
\Hom_{\Lie} (\Gs, \Ufam(\FMod)) \setminus \set{0}\times\FMod
\]
and is constructible. By Richardson's Theorem the sub-levels of the map
from~$\FMod$ to~$S$ defined by~$p\mapsto S(\Ufam_p)$ are closed in the
analytic topology. Hence~$\FMod'$ is closed in the analytic
topology and constructible, it must also be closed in the Zariski topology.
\end{proof}

\begin{corollary}
\label{co:dimalcev}
The map~$\SLie \to S$ sending a Lie algebra to the isomorphism class of its
semi-simple factors is lower semi-continuous, in the Zariski topology.
\end{corollary}


\subsubsection{Proof of~\ref{th:main1} ---  Integration of solvable families}
\label{ss:resolv}

\begin{proposition}
\label{pr:ranksc}
Let~$\FMod$ be a subvariety of~$\SLie$. If~$\FMod$ is contained in~$\SAlg$ and
if each fibre of~$\Ufam(\FMod)$ is a solvable subgroup of~$G$, then the map sending~$p\in\FMod$ to
the rank of~$\Ufam_p$ is lower semi-continuous, in the Zariski topology.
\end{proposition}

We start with two lemmas:

\begin{lemma}
If~$\Gh$ is an algebraic Lie subalgebra of~$\Gg$, then any maximal torus
of~$\Gh$ is algebraic.
\end{lemma}

\begin{proof}
The algebraic Lie algebra~$\Gh$ contains the algebraic hull of any of its Lie
subalgebra, and the algebraic hull of a torus is again a torus.
\end{proof}

\begin{lemma}
Let~$\Gs$ be a semi-simple Lie algebra and~$\Ga$ a subalgebra of a
maximal torus of~$\Gs$. If~$\Ga$ is algebraic, then the restriction to~$\Ga$
of the Killing form of~$\Gs$ is regular.
\end{lemma}

\begin{proof}
Let~$\Phi$ be the roots of~$\Gs$ with respect to a maximal torus of~$\Gs$
containing~$\Ga$. The Killing quadratic form evaluates on~$a\in\Ga$ to
\[
\sum_{\alpha\in\Phi} \alpha^2(a)
.
\]
But~$\Ga$ is algebraic, so there is a basis of it such that
the~$\alpha\in\Phi$ take
integral values on its members. On the rational linear subspace
of~$\Ga$ spanned by this basis, the Killing quadratic form is positive
definite, hence regular. The space~$\Ga$ is the complexification of this
rational linear subspace~$\Ga_\Q$ and the Killing form is then the
complexification of a regular form on~$\Ga_\Q$, it must be regular as well.
\end{proof}

\begin{proof}[Proof of~\ref{pr:ranksc}]
By Ado's Theorem we can assume that our Lie algebras are subalgebras of some
simple Lie algebra~$\Gs$. It follows from the two previous lemmas that
for~$p\in\FMod$ the unipotent radical of~$\Ufam_p$ is the kernel of the
map~$\Ufam_p\to\Ufam_p^*$ defined by the Killing form of~$\Gs$. The conclusion
the follows from Chevalley's Theorem on the dimension of the fibres of a
morphism.
\end{proof}

\begin{theorem}
\label{th:diresolv}
Let~$\FMod$ a locally closed sub-variety of~$\SLie$. If~$\FMod$ is contained
in~$\SAlg$ and each fibre~$\Ufam_p$ is solvable, then the set-theoretic
family~$\UFam(\FMod)$ of subgroups of~$G$ is algebraic.
\end{theorem}

\begin{proof}
We may assume that~$G\cdot \FMod = \FMod$~(\ref{co:fconstr2}
point~$2$). Let~$\Gb$ be a~Borel subalgebra of~$\Gg$. The set~$T$ of
parameters~$p\in\FMod$ such that~$\Ufam_p$ is contained in~$\Gb$ is
constructible and meets all orbit of~$\FMod$ under the operation
of~$G$. By~\ref{co:fconstr2} point~3, we only need to prove the proposition in
the case where~$\FMod = T$. The Lie subalgebra~$\Gb$ of~$G$ is algebraic,
let~$B$ be an algebraic subgroup of~$G$ such that~$\Lie(B) = \Gb$, note
that~$B\cdot \FMod = \FMod$.

Let~$\Gs$ be a maximal torus of~$\Gb$,~$\Gu$ its unipotent radical, and
let~$m$ be the maximal rank of the solvable algebras~$\Ufam_p$
for~$p\in\FMod$. Since the rank is semi-continuous~(\ref{pr:ranksc}) we may
use~\ref{co:fconstr2} a last time to restrict ourselves to the case where
each~$p$ in~$\FMod$ corresponds to a fibre~$\Ufam_p$ meeting~$\Gs$ along a
$m$-dimensional subspace.

Since~$\Gs \cap \Ufam_p$ is an algebraic subtorus of~$\Gs$, it can not move
and there is a subtorus~$\Gs'$ of~$\Gs$ such that~$\Gs \cap \Ufam_p = \Gs'$
for all~$p\in \FMod$ and~$S'$ a connected subgroup of~$G$ whose Lie algebra
is~$\Gs'$.

Now we denote by~$\Ufam^u(\FMod)$ the family of unipotent radical of the
fibres of~$\Ufam(\FMod)$, it is defined by
\[
\Ufam^u(\FMod) = \Ufam(\FMod) \cap (\Gu \times \FMod)
.
\]
The exponential map restricts to an algebraic morphism on~$\Gu$, thus
\[
\UFam^u(\FMod) (\exp\times\id_\FMod(\Ufam^u(\FMod)))
\]
is an algebraic family of subgroups of~$G$. We conclude with~\ref{co:fconstr3}
that
\[
\UFam_p = \UFam^u_p   S'
\]
defines an algebraic family of subgroups of~$G$ parametrised
by~$\FMod$, and~$\Lie(\UFam(\FMod)) = \Ufam(\FMod)$.
\end{proof}


\subsubsection{Proof of~\ref{th:main1} ---  Integration of general families}

\begin{proposition}
\label{pr:digeneral}
Let~$\FMod$ be a quasi-projective
subvariety of~$\SLie$. If~$\FMod$ is contained in~$\SAlg$
then the set-theoretic family~$\UFam(\FMod)$ is an algebraic subvariety of~$G
\times \FMod$.
\end{proposition}

\begin{proof}
Taking~\ref{co:fconstr2}, point~1 and~\ref{co:dimalcev} into account, we may
freely assume that all the semi-simple factors of the fibres~$\Ufam(\FMod)$
are isomorphic to a fixed semi-simple algebra~$\Gs$. Moreover all semi-simple
factors of a Lie algebra are equivalent under the operation of the unipotent
isomorphisms of this algebra. We may invoke~\ref{pr:fconstr}, point~3 to
restrict ourselves to the case where~$\Gs$ is a subalgebra of each fibre
of~$\Ufam(\FMod)$. The unipotent radical of a Lie algebra is the kernel of its
Killing form, hence the set-theoretic family~$\Ufam^r(\FMod)$ of Lie
subalgebras of~$\Gg$ whose fibre at~$p$ is the unipotent radical of~$\Ufam_p$
is algebraic. By~\ref{th:diresolv} there is an algebraic family~$\FFam$ of
connected subgroups of~$G$ parametrised by~$\FMod$ such that~$\Lie(\FFam) =
\Ufam^r(\FMod)$. Let~$S$ be the connected subgroup of~$G$ whose Lie algebra
is~$\Gs$, we conclude by~\ref{co:fconstr3} that
\[
\UFam(\FMod) = (S\times\FMod) \FFam
,
\]
hence it is algebraic.
\end{proof}


\section{Moduli space and universal family}
\label{se:mo}
\subsubsection{Bouquet of algebraic varieties}
\label{ss:mobouquet}

The moduli space we construct is not an algebraic variety, but is a ringed
space obtained by gluing together possibly infinitely many
algebraic varieties along closed subsets.

\begin{definition}
Let~$\mathcal{X}$ be a set of irreducible algebraic varieties. A \emph{gluing
pattern}~$(F,h)$ for~$\mathcal{X}$ assigns to each pair~$(X,Y)$ of
elements of~$\mathcal{X}$ a closed subset~$F_{YX}$ of~$X$ and an
isomorphism~$h_{YX}: F_{YX} \to F_{XY}$, tied to satisfy the following
properties:
\begin{enumerate}
\item
For all~$x\in\mathcal{X}$,
one has~$f_{XX} = X$ and~$h_{XX}$ is the identity map.
\item
For any triple~$(X,Y,Z)$ of elements of~$\mathcal{X}$, the
composition~$h_{ZY} h_{YX}$ agrees with~$h_{ZX}$ whenever it can be
defined.
\end{enumerate}
\end{definition}
A gluing pattern on a set of algebraic varieties defines an equivalence
relation on the disjoint sum of the elements of~$\mathcal{X}$. The quotient
space~$\Bouquet$ of this equivalence relation carries a natural structure 
of ringed space.

The Zariski topology on~$\Bouquet$ is the final topology on the
set~$\Bouquet$, with respect to the natural injections~$X\to\Bouquet$.
Since the various~$F_{YX}$ are closed, the maps~$X\to\Bouquet$ are closed
immersions, the set of irreducible components of~$\Bouquet$ is~$\mathcal{X}$,
and a subset~$U$ of~$\Bouquet$ is open if, and only if,
for all~$X\in\mathcal{X}$, the trace~$U\cap X$ is open in~$X$.

If~$U$ is such an open subset of~$\Bouquet$, the ring~$\Regular(U)$ of regular
functions on~$U$ is the subring of the ring
\[
\prod_{X\in\mathcal{X}} \Regular_X(U \cap X)
\]
spanned by the elements that define a function on the
quotient~$\Bouquet$.

\begin{definition}
A \emph{bouquet of algebraic varieties} is a ringed space~$\Bouquet$.
\end{definition}

The name is suggested by the analogy between the gluing pattern of the moduli
space for 1-dimensional connected subgroups of~$\SL_3$ described in~\ref{ss:exsl3k1}
and the one of a \emph{bouquet of circles} in Topology.

Defining a morphism between two bouquets is the same as
defining a set of morphisms between their irreducible components, that are
compatible with their gluing patterns. In the category of bouquets, one can
again glue together a family of bouquets, and obtain a new bouquet.

Note that if~$E$ is a vector space and~$k$ an integer, the tautological bundle
over the Grassmann variety of $k$-dimensional subspaces of~$E$ is a final
object in the category of flat families of $k$-dimensional subspaces of~$E$
over a bouquet.

\begin{remark}
Bouquets of algebraic varieties are called~\emph{infinite-dimensional
varieties} by~Shafarevich~\cite{IND1} or~\emph{ind-varieties}
by~Kumar~\cite{IND2}. However, these authors seem to have been interested by
the case where the variety obtained after gluing is irreducible and infinite
dimensional at its general point, while the varieties we are interested in are
finite dimensional at their general point,
and have infinitely many irreducible components. We thus felt that it would be
misleading to refer to the varieties we encountered under one of the
afore mentioned names.
\end{remark}


\subsubsection{Construction}

\begingroup
\newcommand\XUMod{\mathcal{X}_\UMod}
\newcommand\GUMod{(A,h)_\UMod}
\newcommand\BUMod{\XUMod/\GUMod}
\newcommand\XUFam{\mathcal{X}_\UFam}
\newcommand\GUFam{(A,h)_\UFam}
\newcommand\BUFam{\XUFam/\GUFam}

Let~$\XUMod$ be the set of all irreducible subvarieties of~$\SLie$
that are contained in~$\SAlg$ and~$\XUFam$ the set
\[
\XUFam = \set{
\UFam(\FMod) \mid \FMod \in \XUMod
}
\]
of irreducible subvarieties of~$G \times \SLie$~(\ref{th:main1}). We define
a gluing pattern~$\GUMod$ (resp.~$\GUFam$) for~$\XUMod$ (resp.~$\XUFam$) by
assigning to each pair~$(\FMod_1,\FMod_2)$ of elements of~$\XUMod$
(resp.~$\XUFam$) the closed subset~$\FMod_1 \cap \FMod_2$ of~$\FMod_1$ and the
morphism from~$\FMod_1 \cap \FMod_2$ to~$\FMod_2 \cap \FMod_1$ obtained by
restricting the identical transform of~$\SLie$ (resp.~$G\times\SLie$).

\begin{definition}
The bouquet of algebraic varieties~$\UMod = \BUMod$ is the~\emph{moduli space
for connected subgroups of~$G$} and~$\UFam = \BUFam$ is the~\emph{universal
family for connected subgroups of~$G$.}
\end{definition}

The universality of these objects lies in the following

\begin{theorem}
\label{th:main2}
Let~$\FMod$ be an algebraic variety, and denote by
\begin{enumerate}
\item
$F$ the set of flat families of $k$-dimensional connected subgroups of~$G$ parametrised
by~$\FMod$;
\item
$M$ the set of morphisms of ringed spaces from~$\FMod$ to~$\UMod$.
\end{enumerate}
Then the map~$F\to M$ sending a family~$\FFam$ to the morphism to~$\SLie$
associated with~$\Lie(\FFam)$ and the map from~$M \to F$
sending a morphism~$\psi$ to the family~$\psi^*\UFam$ define a natural
correspondance between~$F$ and~$M$.
\end{theorem}

\begin{remark}
One can introduce the category of flat families of connected subschemes of~$G$
parametrised by bouquets of algebraic varieties. It follows from
Theorem~\ref{th:main2} that~$\UFam\to\UMod$ is a final object in this category. 
\end{remark}

\endgroup


\section{Family of identity components}
\label{se:ic}
The set of connected components of an algebraic group is
naturally identified with the quotient group of the group by its identity
component. We show that the set-theoretic family~$\FFam^o$ of
the identity components of the fibres of~$\FFam$ is an algebraic
family, and the family of quotients~$\FFam/\FFam^o$ can be constructed
for the general parameter.

\subsubsection{Closedness}

\begin{proposition}
\label{pr:clcc}
Let~$G$ be an algebraic group, $\FMod$ a quasi-projective variety,
and~$\FFam$ a family of loose subgroups of~$G$ parametrised
by~$\FMod$. For each~$p\in\FMod$ there is a closed subset~$F_p$
of~$\FFam$ that contains all the identity components of the
fibres of~$\FFam$ and meets~$\FFam_p$ exactly along its identity
component~$\FFam_p^o$.
\end{proposition}

\begin{proof}
Let~$K_p$ be any projective
completion of~$G$ such that the connected components of the
closure of~$\FFam_p$ in~$K$ are the closures of the connected components
of~$\FFam_p$~(see~\ref{pr:compactification}), and let~$Q$ be any projective
completion of~$\FMod$. We consider the inclusion of~$G \times \FMod$ in~$K
\times Q$ and the Stein
factorisation of the projection to~$Q$ of the closure~$X$ of the image
of~$\FFam$ in~$K\times Q$:
\[
X \labelledto{f} Q' \to Q
.
\]
Now let~$F_p = \FFam \cap \inv f(f(\set{\theta_p(e)}\times Q))$, this is
a closed subset of~$\FFam$ that contains the identity component of
each fibre of~$\FFam$ and meets~$\FFam_p$ precisely along~$\FFam_p^o$.
\end{proof}

\begin{corollary}
\label{co:clcc}
Let~$G$ be an algebraic group, $\FMod$ a variety,
and~$\FFam$ a family of (resp.~loose) subgroups of~$G$ parametrised
by~$\FMod$. The set-theoretic family
\[
\FFam^o = \set{ (g,p) \in \FFam \mid g \in \FFam_p^o }
\]
of the identity components of the fibres of~$\FFam$ is closed
in~$\FFam$ and defines an algebraic family of (resp.~loose) subgroups of~$G$
parametrised by~$\FMod$.
\end{corollary}

\begin{proof}
Being closed is a local property, hence we may assume that~$\FMod$ is
an affine variety. By~\ref{pr:clcc} we have
\[
\FFam^o = \bigcap_{p \in \FMod} F_p
\]
which shows that~$\FFam^o$ is a closed subset of~$G \times \FMod$.
Since each fibre~$\FFam^o_p$ of~$\FFam^o$ is open in~$\FFam_p$, the former is
reduced when the latter is.
\end{proof}

\subsubsection{Quotient}

\begin{proposition}
\label{pr:fcompact}
Let~$G$ be an algebraic group and~$\FMod$ be an affine variety. For all
family~$\FFam$ of loose subgroups of~$G$ parametrised by~$\FMod$,
there exists a projective variety~$X$, an embedding of~$G\times\FMod$ in~$X
\times\FMod$ and an open subset~$U$ of~$\FMod$ such that for
all~$p\in U$ the connected components of the closure of~$\FFam_p$
in~$K$ are the closures of the connected components of~$\FFam_p$.
\end{proposition}

In other words, there exists a compactification of~$\FFam$ such that the
various connected components of a general fibre in~$\FFam$ have no common
points at infinity.

\begingroup
\newcommand\CQ{\mathcal{Q}}
\begin{lemma}
Let~$L$ be an affine group and~$\FMod$ an affine variety. For all
family~$\FFam$ of loose subgroups of~$L$ parametrised by~$\FMod$,
there is
\begin{enumerate}
\item
a finite partition~$\CQ$ of~$P$;
\item
representations~$(V_Q)_{Q\in\CQ}$ of~$G$;
\item
morphisms~$d_Q: \FMod\to V_Q$;
\end{enumerate}
such that for all~$Q\in\CQ$ and all~$p\in Q$ we have~$d_Q(p)\neq0$ and the
stabiliser of~$[d_Q(p)]\in\P(V_Q)$ in~$G$ is~$\FFam^o_p$.
\end{lemma}

\begin{proof}
Let~$E$ be a finite dimensional $L$-stable linear subspace of~$\Reg_L \otimes
\Reg_\FMod$ containing generators of the ideal~$I$ of the family~$\FFam^o$ of
identity components of~$\FFam$, and let~$W = E \cap I$. For each~$p$
in~$\FMod$, we denote by~$\epsilon(p)$ the partial evaluation morphism
from~$\Reg_L \otimes
\Reg_\FMod$ to~$\Reg_L$ sending a function~$f$ to the
function~$\epsilon(p)(f)$ defined
by~$\epsilon(p)(f)(g) = f(g,p)$. The linear subspace~$E =
\epsilon(\FMod)(E)$ of~$\Reg_L$ spanned by partial evaluations is finite
dimensional, and contains all the spaces
\[
W_p = \epsilon(p)(W),\qquad p\in \FMod
.
\]
Let~$\delta:\FMod\to\N$ be the function defined by~$p\mapsto\dim W_p$ and
put~$m = \max\set{\delta(\FMod)}$. For each~$p$ in~$\FMod$ 
there is a subsystem of~$(\epsilon(p)(w_1), \dots, \epsilon(p)(w_m))$
that is a basis of~$W_p$. Let~$\gamma$ be a function choosing for
each~$p\in\FMod$ indices in~$\set{1,\dots,m}$ of such a subsystem.
Let~$\CQ$ be the partition of~$\FMod$ defined by the
function~$\delta\times\gamma$.
For each~$Q\in\CQ$ we set
\[
V_Q
=
\Alt^{\delta(Q)} E
\]
and, with~$\gamma(Q) = \set{ i_1 <\cdots< i_{\delta(Q)} }$,
\[
d_Q (p) =
\epsilon(p)(w_{i_1})\wedge \cdots \wedge \epsilon(p)(w_{i_{\delta(Q)}})
.
\]
When~$p\in\FMod$ belongs to~$Q\in\CQ$, the map~$d_Q$ does not vanish
at~$p$. The ideal of~$\Reg_G$ spanned by
\[
\set{\epsilon(p)(w_{i_1}),\dots, \epsilon(p)(w_{i_{\delta(Q)}})}
\]
is the ideal of~$\FFam^o_p$, hence the stabiliser of~$[d_Q(p)]\in\P(V_Q)$
in~$G$ is~$\FFam^o_p$ (see~\cite[3.8 and~5.1]{BOREL}).
\end{proof}

\begin{proof}[Proof of~\ref{pr:fcompact}]
According to the Chevalley structure Theorem~\cite{CONRAD}, $G$ is an extension of
an abelian variety~$A$ by an affine group~$L$. We apply the lemma to~$L$ and
put
\[
V = \bigoplus_{Q\in\CQ} V_Q
\quad\text{and}\quad
d = \bigoplus_{Q\in\CQ} d_Q
.
\]
We denote by~$\pr_Q$ the projection from~$\P(V)\times\P(V_Q)$ to~$\P(V_Q)$ of
the closure of the graph of the projection from~$\P(V)$ to~$\P(V_Q)$. The
latter is a rational $G$-invariant map that is defined at~$[d(p)]$ for
each~$p\in Q$.

Starting with~$K_0$ any projective
compactification of~$G$ we let~$X = K_0\times\P(V)$ and define the embedding
\[
\fun\theta\: G\times \FMod\to X \times \FMod\\
(g,p)\mapsto(g, g d(p), p)
.\\
\]
Now choose~$p\in\FMod$, $Q\in\CQ$ containing~$p$ and consider two different connected
components~$h_1\FFam_p^o$ and~$h_2 \FFam_p^o$, where~$\set{h_1, h_2}\subset
L$, of~$\FFam_p$. The stabilizer of~$[d_Q(p)]$ in~$L$
is~$L\cap\FFam_p^o$~(\ref{co:appcomp2}), hence the two
subvarieties
\[
h_1\FFam_p^o d_Q(p)
\quad\text{and}\quad
h_2\FFam_p^o d_Q(p)
\]
of~$\P(V_Q)$ are disjoint. Since they are isomorphic to the abelian
variety~$\FFam_p^o/(\FFam_p^o\cap L)$, they are furthermore closed
in~$\P(V_Q)$.  We conclude that~$h_1\FFam_p^o$ and~$h_2\FFam_p^o$ are
respectively mapped by~$(\id_G\times \pr_Q \times \id_\FMod)\circ\theta$ into
the two disjoint closed subvarieties
\[
K_0 \times \set{[h_1 \FFam_p^o d_Q(p)]} \times\set{p}
\quad\text{and}\quad
K_0 \times \set{[h_2 \FFam_p^o d_Q(p)]} \times\set{p}
\]
of~$G\times\P(V_Q)\times\FMod$, so their closures in~$X \times\FMod$
remain disjoint.
\end{proof}
\endgroup

\begin{corollary}
\label{co:stein}
Let~$G$ be an algebraic group, $\FMod$ a variety,
and~$\FFam$ a family of loose subgroups of~$G$ parametrised
by~$\FMod$. The projection~$\FFam \to\FMod$ admits a Stein factorisation
\[
\FFam \labelledto{\pi_0} \pi_0(\FFam) \to \FMod
\]
of~$\FFam\to P$: the morphism~$\pi_0(\FFam)$ is an algebraic variety, the
morphism~$\FFam\to\pi_0(\FFam)$ has connected fibres and the
morphism~$\pi_0(\FFam)\to\FMod$ is proper and finite.
\end{corollary}

The fibre of~$\pi_0(\FFam)$ at~$p\in\FMod$ is the group~$\pi_0(\FFam_p)$, so
we were able to perform the quotient~$\FFam_p/\FFam_p^o$ in all the fibres of
the family~$\FFam$ simultaneously.

\begin{proof}
Let~$j:\FFam\to X\times\FMod$ the embedding provided
by~\ref{pr:fcompact}, it is enough to define~$\pi_0$ on an
open cover of~$\FMod$ so we may restrict ourselves to the case where~$\FMod$ is
affine, and admits a projective compactification~$Q$. Now the wished Stein
factorisation is obtained by restricting to~$\FFam\to \FMod$ the Stein
factorisation of the projection on~$Q$ of the closure of~$j(\FFam)$ in~$X
\times Q$.
\end{proof}


\section{Examples}
\subsubsection{Families of subgroups of tori}

Let~$S$ be a torus. A 1-parameter subgroup~$\lambda$ of~$S$ is characterized
by its initial tangent vector~$\lambda'(1)\in\Gs$. The correspondence between
1-parameter subgroups and initial tangent vectors induces
a group isomorphism between the group of all 1-parameter subgroups of~$S$
and the lattice~$Y \subset \Gs$ of their initial tangent vectors. A~connected
subgroup of~$S$ is a torus as well, and is characterized by the set of
its 1-parameter
subgroups, hence we have a natural bijection between the set~$\SGrp$ of
$k$-dimensional connected subgroups of~$S$ and~$\Grass(k,Y)
\subset\Grass(k,\Gs)$.

\begin{proposition}
Let~$S$ be a torus,~$\FMod$ a variety and~$\FFam$ a family of loose
subgroups of~$S$ parametrised by~$\FMod$. If each irreducible component
of~$\FFam$ dominates~$\FMod$, then~$\FFam$ is a trivial family.
\end{proposition}

\begin{proof}
We first assume that~$\FFam$ is irreducible, so that each fibre is
connected~(\ref{co:clcc}).
Let then~$\Lie(\FFam) \subset \Gs\times\FMod$ be the algebraic family whose
fiber at~$p \in \FMod$ is the Lie algebra of~$\FFam_p$. This family is
equidimensional, hence there is an open subset~$\FMod^o$ of~$\FMod$
such that its restriction to~$\FMod^o$ is flat. The universal
property of the Grassmann variety yields a morphism~$\psi: \FMod^o \to \SGrs$
such that~$\Lie(\FFam)$ is the pullback of the tautological bundle
on~$\SGrs$. But~$\psi$ takes its values in~$\Grass(k,Y)$ which, in the euclidean
topology, is a totally discontinuous topological space. Hence~$\psi$
is constant and~$\FFam$ is a trivial family above~$\FMod^o$, it must
then be trivial above all of~$\FMod$.

We now treat the case where~$\FFam$ may have multiple connected
components. By the case where~$\FFam$ is irreducible there is a
subtorus~$S'$ of~$S$ such that~$\FFam^o = S' \times \FMod$. By Baire's
Theorem there is an exponant~$N\in\N$ that kills all the fibres of the
image of~$\FFam$ in~$S/S' \times \FMod$. The subgroup~$C$ of~$S$
consisting of elements killed by~$N$ is finite and if~$C' = \set{ s\in
C \mid \exists p \in \FMod \; s \in \FFam_p }$ then~$\FFam = C' S'
\times \FMod$, so it is a trivial family.
\end{proof}


\subsubsection{Families of subgroups of abelian varieties}

A variation~(\ref{le:arigidity}) on the \emph{rigidity lemma} for projective
morphisms~\cite{MUMFORD} and the Stein factorization enable us to show that
families of subgroups of abelian varieties are constant~(\ref{pr:arigidity}).

\begin{proposition}
\label{co:aclcc}
Let~$A$ be an abelian variety,~$\FMod$ a variety and~$\FFam$ a family of (loose)
subgroups of~$A$ parametrised by~$\FMod$. The set-theoretic family
\[
\FFam^o = \set{ (a,p) \in A \times \FMod \mid a \in \FFam_p^o }
\]
of subgroups of~$A$ is an algebraic family of (loose) subgroups of~$A$.
\end{proposition}

\begin{proof}
We show that~$\FFam^o$ is a closed subset of~$\FFam$. Since being closed
in~$A \times \FMod$ is a local property, we may assume that~$P$ is affine, and
therefore admits a projective completion~$X$.

Now consider the Stein factorisation
\[
\bar\FFam \labelledto{f} X' \to X
\]
of the projection to~$X$ of the closure~$\bar\FFam$ of~$\FFam$ in~$A \times
X$: the variety~$X'$ is projective, the morphism~$f$ has connected fibers
and~$X' \to X$ is finite. We have then
\[
\FFam^o = \FFam \cap \inv f(f( \set{e}\times X),
\]
which shows that this set is closed. For each~$p$, the map~$\FFam_p^o \to
\FFam_p$ is an open immersion, hence each~$\FFam_p^o$ is reduced when
each~$\FFam_p$ is.
\end{proof}

\begin{corollary}
\label{co:aconnected}
Let~$A$ be an abelian variety,~$\FMod$ a variety and~$\FFam$ a family of (loose)
subgroups of~$A$ parametrised by~$\FMod$. If~$\FFam$ is irreducible, then each
of its fibre is connected.
\end{corollary}

\begin{corollary}
\label{co:aconnirr}
Let~$A$ be an abelian variety,~$\FMod$ a variety and~$\FFam$ a family of loose
subgroups of~$A$ parametrised by~$\FMod$. If each irreducible component
of~$\FFam$ dominates~$\FMod$, then~$\FFam^o$ is irreducible.
\end{corollary}

\begin{proof}
The fibres of~$\FFam^o\to\FMod$ are irreducible, let~$n$ be their minimal
dimension. By Chevalley's Theorem for the dimensions of the fibres of a
morphism, the general fibre of~$\FFam^o\to\FMod$ has dimension~$n$. There
exists an irreducible component~$\FFam^1$ of~$\FFam^o$ whose general fibre has
dimension~$n$. Hence~$\FFam^1$ and~$\FFam^o$ agree above a dense open subset
of~$\FMod$. Since each irreducible~$\FFam^2$ component of~$\FFam^o$
dominates~$P$, it meets~$\FFam^1$ along a dense open subset of~$\FFam^2$, so
that~$\FFam^2 \subset\FFam^1$. We conclude that~$\FFam^o$ is irreducible.
\end{proof}

\begin{proposition}
\label{pr:arigidity}
Let~$A$ be an abelian variety,~$\FMod$ a variety and~$\FFam$ a family of loose
subgroups of~$A$ parametrised by~$\FMod$. If each irreducible component
of~$\FFam$ dominates~$\FMod$, then~$\FFam$ is a trivial family.
\end{proposition}

We start with two lemmas:

\begin{lemma}[Rigidity lemma variant]
\label{le:arigidity}
Let~$X$ be a complete variety, $Y$ and~$Z$ any varieties, $\VFam \subset X
\times Y$ an irreducible family of connected subvarieties of~$X$ parametrised
by~$Y$ admitting a section, and~$f: \VFam \to Z$ a morphism. If there exists a
member
of~$\VFam$ contracted to a point by~$f$, then there is a morphism~$g: Y \to Z$
such that~$f = g \circ \pr_Y$, where~$\pr_Y$ is the projection of~$\VFam$
to~$Y$.
\end{lemma}

\begin{proof}
Let~$s: Y \to \VFam$ be a section of~$\VFam$ and let~$g: Y \to Z$ be the map
defined by~$g(y) = f(s(y),y)$. Since~$\VFam$ is irreducible, we only need to
show that~$f$ and~$g \circ \pr_Y$ agree on a dense open subset
of~$\VFam$.

Let~$y_0$ be a parameter such that~$\VFam_{y_0}$ is contracted to
a point~$z_0$ by~$f$, and let~$U$ an affine neighbourhood of~$z_0$
in~$Z$. Let~$F = Z \setminus U$ its complement and~$G = \pr_Y \inv
f(F)$. Since~$\inv f(F)$ is a closed subset of~$X \times Y$ and~$X$ is
complete, the projection~$G$ of~$\inv f(F)$ on~$Y$ is closed. Further~$y_0
\not\in G$ since~$f(\VFam_{y_0}) = { z_0 }$. Therefore~$V = Y - G$ is a
non-empty open subset of~$Y$, for each~$y \in V$ the complete
connected variety~$\VFam_y$ gets mapped by~$f$ into the affine variety~$U$,
hence to a single point of~$U$. But this means that for any~$x \in X$, $y\in
V$, we have~$f(x,y) = f(s(y),y) = g \circ \pr_Y(x,y)$.
\end{proof}

\begin{lemma}
\label{le:aconstant}
Let~$A$ be an abelian variety,~$\FMod$ a variety and~$\FFam$ a family of loose
subgroups of~$A$ parametrised by~$\FMod$. If~$\FFam$ is irreducible, then it
is a trivial family.
\end{lemma}

\begin{proof}
Since~$\FFam$ is irreducible, it has connected fibres~(\ref{co:aconnected}).
Let~$p$ be a point of~$\FMod$ dominated by a connected
fiber~$\FFam_p$ of~$\FFam$. According to a theorem of
Chow~\cite[Theorem~1]{CHOW} there exists an abelian variety~$Z$ and a group
homomorphism~$f_0: A \to Z$ whose kernel is~$\FFam_p$. Let~$f$ be the
restriction to~$\FFam$ of~$f_0 \times \id_\FMod$. The fiber at~$p$ is
contracted by~$f$ to the unit element of~$Z$, hence it follows from
Lemma~\ref{le:arigidity} that~$f$ can be factorized out through the
projection~$pr_\FMod$ of~$\FFam \to \FMod$. Consequently, each fiber~$\FFam_q$
of~$\FFam$ is contained in the kernel~$\FFam_p$ of~$f_0$. We conclude with the
theorem on the dimension of the fibers of a morphism that there is an open
subset~$V$ of~$\FMod$ above which  each fiber actually
equals~$\FFam_p$. But~$\FFam$ is irreducible and equals the closure
of~$\FFam_p \times V$ in~$A \times \FMod$ that is, that the family~$\FFam$ is
trivial.
\end{proof}

\begin{lemma}
\label{le:afinite}
Let~$A$ be an abelian variety,~$\FMod$ a variety and~$\FFam$ a family of loose
finite subgroups of~$A$ parametrised by~$\FMod$. If each irreducible component
of~$\FFam$ dominates~$\FMod$, then~$\FFam$ is a trivial family.
\end{lemma}

\begin{proof}
By Baire's Theorem, there is a common exponent~$N\in\N$ that kills all the
fibres of~$\FFam$. But the subgroup~$B$ of~$A$ consisting of its elements of
order~$N$ is finite, and~$\FFam$ is a subvariety of~$B \times \FMod$. Since
each irreducible component of~$\FFam$ dominates~$\FMod$, this subvariety is of
the form~$B'\times \FMod$ for some subgroup~$B'$ of~$B$.
\end{proof}

\begin{proof}[Proof of~\ref{pr:arigidity}]
By~\ref{co:aconnirr} the family~$\FFam^o$ of connected components of~$\FFam$
is irreducible, and by~\ref{le:aconstant} it must be trivial. Let~$B$ be the
connected closed subgroup of~$A$ such that~$\FFam^o = B \times
\FMod$. According to Chow~\cite[Theorem~1]{CHOW} there is an abelian
variety~$Z$ and a morphism of algebraic groups~$f:A\to Z$ whose kernel is~$B$.

We denote by~$f_\FMod$ the morphism~$A \times\FMod \to Z \times \FMod$
associated to~$f$ and by~$f_\FMod(\FFam)$ the image of~$\FFam$ through this
morphism. It is a closed subvariety of~$Z \times \FMod$, and defines a family
of loose finite subgroups of~$Z$ whose irreducible components
dominate~$\FMod$. We conclude by~\ref{le:afinite} that his family is trivial,
of the form~$B'\times \FMod$. Hence~$\FFam = \inv f(B') \times \FMod$ is also a
trivial family.
\end{proof}


\subsubsection{Universal family for~$\SL_2$}

We proceed in two steps, analysing first the one-dimensional subgroups and
second the two-dimensional subgroups. We let~$G = \SL_2$.

\paragraph{One-dimensional subgroups}
The operation of~$G$ on~$\GP(\Gg) =
\SGrs$ has two orbits: the closed orbit is the set of nilpotent elements
in~$\GP(\Gg)$ and the open one is the set of semi-simple elements. Thus
each one-dimensional Lie subalgebra of~$\Gg$ is algebraic: $\SAlg_1 =
\GP(\Gg)$ and the universal family of one-dimensional subgroups of~$G$
is a smooth irreducible variety of dimension~$3$.

\paragraph{Two dimensional subgroups}
A~2-dimensional subgroup of~$G$
is a Borel subgroup~$B$. According to the theory of reductive groups, their
set~$\SGrp$ is the variety~$G/B\simeq\P^1$. The universal family of
2-dimensional subgroups of~$G$ is the
saturation with respect to the operation of~$G$ of the set~$\set{ (B, g) \mid
g \in B}$ in~$G/B \times G$. This is a smooth irreducible algebraic variety of
dimension~$3$.


\subsubsection{Universal family of one-dimensional subgroups of~$\SL_3$}
\label{ss:exsl3k1}

This example demonstrates the hairy structure of the moduli space, while
remaining tractable. We let~$G = \SL_3$ and describe the moduli space of
one-dimensional subgroups of~$G$.

Our first task is to describe~$\SAlg_1$ as a set, to do this we take
advantage of the operation of~$G$ on~$\SLie = \GP(\Gg)$. It follows
from the Chevalley-Jordan decomposition that the tangent space to a
one-dimensional subgroup of~$G$ is either spanned by a semi-simple
element of~$\Gg$ or by a nilpotent element. The group~$G$ has two
nilpotent orbits~$\set{\Nilpotent^6, \Nilpotent^3}$ in~$\GP(\Gg)$
whose dimensions are respectively~$6$ and~$3$, each of them is
included in~$\SAlg$. In order to parametrize the semi-simple orbits
of~$G$ in~$\SAlg$, we introduce a torus~$S$ of~$G$
and~$\SAlg(S)_1 \subset \SAlg_1$ the countable set of the algebraic Lie
1-dimensional subalgebras of~$\Lie(S)$. Since any semi-simple element
in~$\GP(\Gg)$ is~$G$-conjugated to an element of~$\Lie(S)$, we may
conclude that 
\[
\set{ Gp \mid p \in \SAlg(S) } \cup \set{\smash{\Nilpotent^6,\Nilpotent^3}}
\]
is the partition of~$\SAlg_1$ in~$G$-orbits. An irreducible subvariety
of~$\SLie_1$ contained in~$\SAlg_1$.


\appendix
\newcommand\an[1]{#1^{\mathrm{an}}}
\newcommand\Sheaf{\mathcal{F}}

\section{Observation of flatness through a biholomorphism}
\label{se:flat}

Let~$X$ be an algebraic variety and~$\Reg_X$ its structural sheaf. We denote
by~$\an X$ the corresponding analytic variety and by~$\an\Reg_X$ the sheaf of
holomorphic functions on~$\an X$. Note that~$X$ is endowed with its Zariski
topology while~$\an X$ is with its euclidean one. We denote by~$\phi$ the
natural map~$\an X \to X$. It is is a morphism of ringed spaces, that enables
us to define the analytic extension~$\an\Sheaf$ of a sheaf~$\Sheaf$
of~$\Reg_X$-modules by
\[
\an\Sheaf = \inv\phi \Sheaf \otimes_{\inv\phi\Reg_X} \an\Reg_X
,
\]
this is a sheaf of~$\an\Reg_X$-modules. Serre~\cite{GAGA} introduced the notion
of \emph{flat pair} of rings~(\ibid\ définition~4). The following statement
results immediately from~\cite[Corollaire~1, p.~11]{GAGA}
and~\cite[Proposition~22, p.~36]{GAGA}.

\begin{proposition}
\label{pr:flatpair}
For all sheaf of~$\Reg_X$-module $\Sheaf$ and all~$x\in X$, the natural map~$\Sheaf_x\to
\an\Sheaf_x$ is injective.
\end{proposition}

We use it to compare flatness of~$\Sheaf$ and~$\an\Sheaf$ through a
biholomorphism.

\begin{proposition}
Let~$f:X\to\FMod$ a morphism of algebraic variety and~$\Sheaf$ a sheaf
of~$\Reg_X$-modules on~$X$. For all~$x$, $\Sheaf_x$ is a flat~$(\Reg_\FMod)
_x$-module if, and only if, $\an\Sheaf_x$ is a
flat~$(\an\Reg_\FMod)_x$-module.
\end{proposition}

\begin{proof}
Recall that, when~$A$ is a ring, an~$A$-module $M$ is flat if, and only if,
for any finitely generated ideal~$\Ga$ of~$A$ the natural map~$\Ga\otimes M
\to M$ is injective. 

If~$\Sheaf_x$ is not a flat~$(\Reg_\FMod)_x$-module, the kernel of the
map~$(\Reg_\FMod)_x \otimes \Sheaf_x \to \Sheaf_x$ contains a non-zero
element. This element also belongs to the kernel of the
map~$(\an\Reg_\FMod)_x \otimes \an\Sheaf_x \to \an\Sheaf_x$ and is
non-zero by~\ref{pr:flatpair}, so~$\an\Sheaf_x$ is not a
flat~$(\an\Reg_\FMod)_x$-module.

The converse assertion is implied by the fact that a flat module
remains flat after base extension~\cite[9.1A]{H}.
\end{proof}

\begin{corollary}
\label{co:appflat}
Let~$\FMod$ be a variety and for~$i\in\set{0,1}$ a morphism~$f_i:X_i\to\FMod$,
a sheaf~$\Sheaf_i$ of~$\Reg_{X_i}$-modules, and a point~$x_i\in X_i$. If there
is a local biholomorphism~$\psi: W_0 \to W_1$ mapping a
neighbourhood~$W_0$ of~$x_0$ in~$X_0$ to a neighbourhood~$W_1$ of~$x_1$
in~$X_1$ and such that~$\left.f_1\right\vert_{W_1} = \psi\circ
\left.f_1\right\vert_{W_1}$ and~$\psi^*\left.\an\Sheaf_0\right\vert_{W_0} =
\left.\an\Sheaf_1\right\vert_{W_1}$, then~$\Sheaf_0$ is flat at~$x_0$ if, and
only if,~$\Sheaf_1$ is flat at~$x_1$.
\end{corollary}


\section{Construction of certain compactifications}

\begin{proposition}
\label{pr:rextension}
Let~$X$ be a variety, $G$ be an algebraic group and~$L$ be a subgroup of~$G$
operating on~$X$. There exists a $G$-variety~$G \times_L X$ and a
closed immersion of~$L$-varieties~$j:X \to G \times_L X$ such that,
for all~$G$-variety~$Y$ and all immersion of~$L$-varieties~$j_Y: X \to
Y$, there is a unique morphism of~$G$-varieties~$f: G \times_L X \to
Y$ such that~$j = f \circ j_Y$.

Furthermore, there is a~$G$-invariant projection~$G \times_L X \to
G/L$ whose fibre above the co-set~$L$ is~$j(X)$, and this projection is proper
if, and only if,~$X$ is complete.
\end{proposition}

\begin{proof}
Let~$L$ operate on~$G \times X$ by~$l \cdot (g,x) = (g\inv l, l
x)$. This is a free action and the geometric quotient~$G \times_L X$
exists~\cite{FOGARTY}. It fulfills the conclusions of the proposition.
\end{proof}

\begin{corollary}
\label{co:appcomp1}
When~$X$ and~$G / L$ are complete, so is~$G \times_L X$.
\end{corollary}

\begin{corollary}
\label{co:appcomp2}
If~$\Orbit_1$ and~$\Orbit_2$ are two distinct orbits of~$L$ in~$X$,
then~$G\Orbit_1$ and~$G\Orbit_2$ are two distinct orbits of~$G$ in~$G\times_L
X$.
\end{corollary}

\begin{proof}
Let~$x_1$ and~$x_2$ two points in~$X$ and assume that $Gx_1$ and~$Gx_2$ have a
common point. These two orbits are then equal, and there is~$g\in G$
sending~$x_1$ to~$x_2$. Using the projection~$G\times_L X \to G/L$, one sees
that~$g$ belongs to~$L$, so that~$Lx_1 = L x_2$.
\end{proof}

Let~$G$ be an algebraic group, by~Chevalley structure Theorem~\cite{CONRAD},
there is an exact sequence of algebraic groups
\[
1 \to L \to G \labelledto\alpha A \to 0
\]
where~$L$ is an affine group and~$A$ an abelian variety. The group~$L$ is the
maximal affine subgroup of~$G$. Remark that a subgroup~$H$ of~$G$ is connected
if, and only if,~$\alpha(H)$ and~$H \cap L$ are connected.

\begin{proposition}
\label{pr:compactification}
Let~$G$ be an algebraic group and~$H$ a subgroup of~$G$. There is a
projective completion~$K$ of~$G$ such that the connected components of the
closure of~$H$ in~$K$ are the closures of the connected components of~$H$.
\end{proposition}

\begin{proof}
Let~$L$ be the maximal affine subgroup of~$G$.
By Chevalley's Theorem on homogeneous spaces of affine
groups~\cite[5.1]{BOREL}, there is a linear representation~$E$ of~$L$
containing a line~$D$ whose stabilizer in~$L$ is the identity component~$M^o$
of~$M = L\cap H$. The orbit of~$D$ under the operation of~$M$ on~$\P(E)$ is a
finite subset~$F \simeq M/M^o$ of~$\P(E)$.

Let~$X = G \times_L \P(E)$ and~$K_0$ be any projective completion of~$G$. We
consider the morphism
\[
\fun\theta\: G \to K_0 \times X\\
g\mapsto (g,gD)
\\
\]
and claim that the closure~$K$ of~$\theta(G)$ in~$K_0 \times X$ is a suitable
compactification.

Assume that~$H$ has multiple connected components and let~$h_1 H^o$ and~$h_2
H^o$ be two of them, where~$\set{h_1,h_2}\subset H$. The connected
components~$h_i H^o D$ of the orbit~$H D$ are isomorphic to the abelian
variety~$H^o/M^o$, hence are complete varieties. They are furthermore
disjoint by~\ref{co:appcomp2}. We can
conclude that the connected components~$h_i H^o$ are mapped under~$\theta$
into the disjoint complete subvarieties~$K_0 \times h_i H^o D$ of~$K_0 \times
X$, so their closures are disjoint.
\end{proof}


\bibliographystyle{plain}
\bibliography{modgrp}
\end{document}